\documentclass[leqno]{article}
\usepackage{amsmath}
\usepackage[latin1]{inputenc}
\usepackage{amssymb}
\usepackage{xcolor}
\usepackage[backref=page]{hyperref}
\usepackage{amsthm} 
\usepackage{enumerate}                                        
\usepackage{url}
\usepackage{graphicx}
\usepackage{epsfig}

\newtheorem{theorem}{\bf Theorem}

\newtheorem{corollary}[theorem]{\bf Corollary}
\newtheorem{lemma}[theorem]{\bf Lemma}

\newtheorem{proposition}[theorem]{\bf Proposition}
\newtheorem{definition}[theorem]{\bf Definition}

\numberwithin{equation}{section}
\numberwithin{theorem}{section}
\numberwithin{figure}{section}

\begin{document}
\renewcommand{\thefootnote}{}
\footnotetext{Research partially supported by Ministerio de Ciencia e Innovaci\'on Grants No: PID2020-118137GB-I00/AEI/10.13039/501100011033, PID2021-126217NB-100 and the `Maria de Maeztu'' Excellence
Unit IMAG, reference CEX2020-001105-M, funded by
MCIN/AEI/10.13039/501100011033.} 

\title{On the first Neumann eigenvalue for critical points of a weighted area functional with asymptotically flat ends} 
\author{\text{A.L. Mart\'inez-Trivi\~no}}
\vspace{.1in}
\date{}
\maketitle
{
\noindent  Departamento de Matem\'aticas, Universidad de C\'ordoba, E-14071 C\'ordoba, Spain
}
\begin{abstract}
In the following work, we obtain a lower bound for the first Neumann eingevalue of the drift Laplacian $\Delta^{\varphi}$ for a family of properly embedded $[\varphi,\vec{e}_{3}]$-minimal surfaces in $\mathbb{R}^3$ with concave function $\varphi$ and asymptotically flat ends. As application, we obtain a control on the topology for that surfaces with finite total curvature. The strategy will consists in an integration of the Bouchner's formula by the works of A. Lichnerowicz \cite{L} and S. Brendle, R. Tsiamis \cite{BT} and relate said eigenvalue with the Poincare's constant.
\end{abstract}
\vspace{0.2 cm}

\noindent e-mail: $\text{almartinez@uco.es}$


\everymath={\displaystyle}

\section{Introduction}
In 1983, H. Choi and A. Wang \cite{CW} obtained a lower bound for the first eigenvalue of a  compact orientable embedded minimal hypersurface $\Sigma$ in a compact orientable Riemannian manifold $M$ with positive Ricci curvature. Precisely, if the Ricci curvature of $M$ is bounded from below by a positive constant $k$, then $\lambda_{1}(\Delta)\geq k/2$, where $\lambda_{1}(\Delta)$ is the first Neumann eigenvalue of the Laplacian $\Delta$ of $\Sigma$. The proof consists in an integration of the Bochner's formula applied by A. Lichnerowicz \cite{L} together with Reilly's formula \cite{R}. An analogous technique was used by S. Brendle and R. Tsiamis \cite{BT} to obtain a lower estimate in $\mathcal{H}^1_{loc}(\Sigma)$ for the Dirichlet energy of the  properly embedded self-similar shrinkers $\Sigma$ of the mean curvature flow, namely, hypersurfaces of the Euclidean space $\mathbb{R}^n$ whose mean curvature satisfies $H=\langle x,N\rangle/2$, where $x$ is the identity vector field  and $N$ is the Gauss map of $\Sigma$. Here $\mathcal{H}^{1}_{loc}(\Sigma)$ is the Hilbert space given by 
$$\mathcal{H}^{1}_{loc}(\Sigma)=\bigg\{f\in W^{1,2}_{loc}(\Sigma) \vert \int_{\Sigma}(f^2+\vert\nabla f\vert^2)\, d\Sigma<+\infty\bigg\},$$
where $\nabla$ is the gradient in $\Sigma$, $d\Sigma$ is the induced measure of $\mathbb{R}^3$ and 
whose scalar product is the induced scalar product of functions $L^2(\Sigma)$. Notice that, the family of self-similar shrinkers are minimal hypersurfaces by a conformal changed of the Euclidean metric, namely, the ambient space $\mathbb{R}^n$ with metric $e^{-\vert x\vert^2/4}\langle\cdot,\cdot\rangle$, where $\langle\cdot,\cdot\rangle$ is the usual metric of $\mathbb{R}^n$. 
\\

Physically, the first eigenvalue of the Laplacian can be interpreted as the squared frequency of the first nontrivial vibration of an elastic membrane. In fact, if the density is not uniform, then the weight function affects how the membrane vibrates. The estimate of the first eigenvalue indicates how "easy" or "difficult" it is for oscillations to occur in the given geometry and density. On the other hand, if we consider the weighted heat equation (i.e., a heat equation adapted to the weighted measurement), the first eigenvalue controls the exponential decay rate of the system's energy with time. The larger the first eigenvalue, the faster the heat dissipates in the manifold. See, for instance \cite{IC} and \cite{CH}.
\\

Motivated by these works, we obtain a lower estimate for the first Neumann eigenvalue for minimal surfaces in the $3$-dimensional  Ilmanen's space, namely, $\mathbb{R}^3$ with a conformal change of the metric $\langle\cdot,\cdot\rangle^{\varphi}=e^{\varphi}\langle\cdot,\cdot\rangle$, see \cite{T}. These surfaces are critical points, under normal variations with compact support, of the following weighted area functional
\begin{equation}
\label{area}
\mathcal{A}^{\varphi}(\Sigma)=\int_{\Sigma}\, e^{\varphi}\, d\Sigma
\end{equation}
on isometric orientable Riemannian surfaces $\Sigma$ in a domain of $\mathbb{R}^3$. The function $\varphi\vert_{\Sigma}:\Sigma\rightarrow\mathbb{R}$ is the restriction on $\Sigma$ of a smooth function $\varphi:\mathbb{R}^3\rightarrow\mathbb{R}$ with support in a tubular neighborhood of $\Sigma$ and which only depends of one direction of $\mathbb{R}^3$. Now, take $\{\vec{e}_{i}\}_{i=1,2,3}$ the usual orthonormal frame of $\mathbb{R}^3$, without loss of generality, we will assume that $\varphi$ only depends of the height function $z=\langle x,\vec{e}_{3}\rangle:\Sigma\rightarrow \mathbb{R}$ and then, hereinafter, we will write by $\varphi\vert_{\Sigma}=\varphi:]a,b[\rightarrow\mathbb{R}$ as a smooth function defined on an interval real numbers.  
\

The critical points of \eqref{area} satisfy the following Euler-Lagrange equation given by its mean curvature $H$ as follows
\begin{equation}
\label{defphimin}
H=-\langle\overline{\nabla}\varphi,N\rangle=-\dot{\varphi}\langle N,\vec{e}_{3}\rangle,
\end{equation}
where $\overline{\nabla}$ denotes the usual gradient of $\mathbb{R}^3$ and $(^\cdot)$ stands the classic derivate of $\mathbb{R}$. An orientable surface $\Sigma$ in $\mathbb{R}^3$ is called $[\varphi,\vec{e}_{3}]$-minimal if and only if its mean curvature $H$ satisfies \eqref{defphimin}. From Radon-Nikod\'ym Theorem, we can consider $\mathbb{R}^3$ as a manifold with density $e^\varphi$ since any oriented surface $\Sigma$ has null weighted mean curvature.
\
Nowadays, the current interest in the development of these surfaces concerns numerous and successful advances in Differential geometry on manifolds with density. For instance:
\begin{itemize}
\item The translating solitons by the mean curvature flow when $\varphi=\text{Id}$ is the identity, that is, the vertical translations $t\rightarrow\Sigma+t\vec{e}_{3}$ is a mean curvature flow and then, the normal component of the velocity in direction $\vec{e}_{3}$ coincides with the mean curvature. At this point, we remark the following works: J. Clutterbuck, O. Schn\"urer and F. Schulze \cite{CSS} describe the asymptotic behavior of rotationally symmetric examples of paraboloid type called Bowl translating soliton and catenoid type called Winglike. Precisely, from the works of J. Spruck, L.Xiao \cite{SX} and X. Wang \cite{Wang}, any complete mean convex translating soliton is convex and so, the Bowl translating soliton is the only entire vertical graph translating soliton. In fact, all complete graphs are classified by D. Hoffman, F. Mart\'in, T. Ilmanen and B. White \cite{HIMW}. 

\item The singular $\alpha$-minimal surface when $\varphi(z)=\alpha\log(z)$, with $\alpha\neq 0$ a real constant, which has important physics applications. For instance, for $\alpha=1$, $\Sigma$ is a heavy surface under a gravitational field, that according to works of D. Dunn \cite{Dunn}, F. Otto \cite{Otto}, Poisson \cite{Poisson} and R. L\'opez \cite{RL1} are of importance for the construction of perfect domes. For $\alpha=2$, we recover the classical minimal surface in the hyperbolic space $\mathbb{H}^{3}$ in the half-space model. Moreover, the study of the global geometry of these examples was given by R. L\'opez \cite{RL} and U. Dierkes \cite{D}.

\item The self-similar shrinkers if we take the Gaussian density $\varphi(\vert x\vert^2)=e^{-\vert x\vert^2/2}$ and the so called shrinking, steady or expanding Ricci solitons if the Backry-Emery tensor $\text{Ric}^{\varphi}(\cdot,\cdot)=\lambda\langle\cdot,\cdot\rangle$ for some $\lambda>0$, $\lambda=0$ or $\lambda<0$, respectively. This family of surfaces is currently a fundamental topic of research. See for instance the works of H.D. Cao \cite{Ca} , P. Petersen \cite{PW1,PW2} and  the well-known works of  G. Perelman \cite{GP1,GP2} solving the Poincare's conjeture. Recently and following the same research line, R. L\'opez \cite{RL2} proved that vertical graphs and radial graphs are stable in the Euclidean space with some densities including translators, expanders and singular minimal hipersurfaces.
\end{itemize}

We recommend to the reader the following works \cite{MM,MM1,MM2,MM3,MMJ,MJ} for more details in the recent advances in the theory of $[\varphi,\vec{e}_{3}]$-minimal surfaces.
\\

 In particular, we show the following results, which explain the hypotheses that we will follow throughout this work due to the global geometric properties depend of the analytical properties of $\varphi$. Firstly, from \cite{MM}, there are not compact examples when  $\varphi$ is a diffeomorphim due to the strong maximum principle. Furthermore, from \cite{MM3}, we obtained a Calabi-Bernstein type result prescribing the Gauss map. Precisely, if $\varphi$ has almost quadratic growth with $\ddot{\varphi}\geq 0$, then there are no non-flat properly embedded $[\varphi,\vec{e}_{3}]$-minimal surfaces with bounded lenght of the second fundamental form $\text{sup}_{\Sigma}\vert\mathcal{S}\vert^2<+\infty$ and prescribed Gauss map $N$ satisfying that $\langle N,\vec{v}\rangle\leq -2\epsilon$ on $\Sigma$ for some unit vector $\vec{v}\in\mathbb{R}^3$ orthogonal to $\vec{e}_{3}$ and some $\epsilon>0$. Moreover, under same conditions of $\varphi$, there are no properly embedded $[\varphi,\vec{e}_{3}]$-minimal surfaces in $\Sigma$ whose Gauss map $N$ verifies $\langle N,\vec{e}_{3}\rangle\geq\epsilon>0$ and such that either $\text{inf}_{\Sigma}\dot{\varphi}^2>0$ or $\text{inf}_{\Sigma}\ddot{\varphi}>0$. 
\\

On the other hand, we studied the stability of the mean convex properly embedded $[\varphi,\vec{e}_{3}]$-minimal surfaces in $\mathbb{R}^3$ in \cite{MMJ}. In fact, if the function $\varphi$ verifies that $\dot{\varphi}>0$ and $\ddot{\varphi}\geq 0$, then $\Sigma$ is stable. In this case, the stability means that $\Sigma$ is a locally minimum of the weighted area functional $\mathcal{A}^{\varphi}$ or equivalently, from \cite{CMZ}, it is stable as minimal surface in the Ilmanen's space.  Precisely, the associated Jacobi operator $\mathcal{L}^{\varphi}$ is a gradient Schr\"ondinger operator defined by
$$\mathcal{L}^{\varphi}(\cdot)=\Delta^{\varphi}(\cdot)+(\vert\mathcal{S}\vert^2-\ddot{\varphi}\eta^2)(\cdot),$$
where $\Delta^{\varphi}$ is the drift Laplacian or also called Witten's Laplacian given by
$$\Delta^{\varphi}(\cdot)=\Delta(\cdot)+\langle\nabla\varphi,\nabla(\cdot)\rangle.$$
We recomemend to the reader the works \cite{CMW, DMWW} to see the structure of the nodal sets and spectral theory on $\Delta^{\varphi}$. Moreover, in the same work \cite{MMJ}, we also obtained a Spruck-Xiao's type result for mean convex  properly embedded $[\varphi,\vec{e}_{3}]$-minimal surfaces with Gauss curvature $K$ bounded from below. Namely, under these assumptions, if $\varphi$ is an strictly increasing diffeomorphism with $\ddot{\varphi}\geq 0$ and $\dddot{\varphi}\leq 0$, then $K\geq 0$.  
\\

Due to the generality of this ambient, we focus our attention in a family of surfaces with a control at their ends and whose family of functions $\varphi$ contain the family concave polynomials with quadratic growth. 
\\

As general hyphotesis, we will assume that:
\\

$\Sigma$ will be a properly embedded $[\varphi,\vec{e}_{3}]$-minimal surface (possibly compact without boundary) on $\mathbb{R}^2\times ]a,b[$ whose ends are asymptotically flat with $a,b\in\mathbb{R}$. Moreover, we will assume that the weight function $\varphi\circ\mu:\Sigma\rightarrow\mathbb]0,+\infty[$ verifies
\begin{equation}
\label{segcond}
 z\dot{\varphi}(z)\leq 0 \ \ \text{and} \ \ \ddot{\varphi}+C_{\varphi}(\Sigma)\dot{\varphi}^2\leq 0,
\end{equation}
for some positive constant $C_{\varphi}(\Sigma)>0$ depending only of $\varphi$ and $\Sigma$. Furthermore, in other situations, to assure the existence of the first Neumann eingenvalue of $\mathcal{L}^{\varphi}$,  the lenght of the second fundamental form $\vert\mathcal{S}\vert^2$ must be bounded, that is,
\begin{equation}
\label{S}
\text{sup}_{\Sigma}\vert\mathcal{S}\vert^2<+\infty.
\end{equation}
In our case, we extend $\varphi:\mathbb{R}\rightarrow [0+\infty[$ to a smooth function with  compact support in a neighborhood of $\mathbb{R}$ containing $]a,b[$. Throughout of this work, we denote by $\mathcal{H}^{1,\varphi}_{loc}(\Sigma)$ the following Hilbert space
$$\mathcal{H}^{1,\varphi}_{loc}(\Sigma)=\bigg\{f\in W^{1,2}_{loc}(\Sigma) \, \vert\, \int_{\Sigma} e^{\varphi}(f^2+\vert\nabla f\vert^2)\, d\Sigma<+\infty \bigg\}$$
with scalar product of the space of functions $L^{2,\varphi}(\Sigma)$ with density $e^{\varphi}$. 
\\

The strategy will be to obtain a lower bound for the weighted enery functional associated to the drift Laplacian $\Delta^{\varphi}$ of $\Sigma$ on the space of functions $\mathcal{H}^{1,\varphi}_{loc}(\Sigma)$. In this case, the problem of eigenvalues of $\mathcal{L}^{\varphi}$ is taken over functions of $\mathcal{H}^{1,\varphi}_{loc}(\Sigma)$ with compact support in domains of $\Sigma$ with (at least) $C^1$ boundary, denoted by $\mathcal{H}^{1,\varphi}_{c}(\Sigma)$. Thus, $\mathcal{L}^{\varphi}$ has discrete spectrum and we characterize the first eigenvalue as 
$$\lambda_{1}(\mathcal{L}^\varphi)=\text{inf}_{\mathcal{H}^{1,\varphi}_{c}(\Sigma)}\frac{\int_{\Sigma}e^{\varphi}\left(\vert\nabla u\vert^2-qu^2 \right)\, d\Sigma}{\int_{\Sigma}e^{\varphi}u^2\, d\Sigma}.$$
The key step will be to relate said eigenvalue with the Poincare's constants of the two connected components of $\mathbb{R}^3\backslash\Sigma$. 
\\

Precisely, our main results stablish the following Faber-Krahn type inequalities
\begin{theorem}[Theorem A]
\label{mainresult}
Let $\Sigma$ be a properly embedded $[\varphi,\vec{e}_{3}]$-minimal surfaces in $\mathbb{R}^2\times ]a,b[$ with asymptotic flat ends and weight function $\varphi$ verifying \eqref{segcond}. Then, for any $f\in\mathcal{H}^{1}_{0}(\Sigma)$ with compact support in $\Sigma\cap K$ for some domain $K$ in $\mathbb{R}^3$ and such that $$\int_{\Sigma}e^{\varphi}f^2\, d\Sigma=1, \ \ \ \ \int_{\Sigma}e^{\varphi}f\, d\Sigma=0 \ \ \text{and} \ \ \int_{\Sigma}e^{\varphi}\vert\nabla f\vert^2\, d\Sigma<+\infty,$$ we get that the following inequality holds
\begin{equation}
\label{mainineq}\int_{\Sigma}e^{\varphi}\, \vert\nabla f\vert^2\, d\Sigma\geq \frac{3}{5}C_{\varphi}(\Sigma)\left(\frac{1}{\text{Vol}(K) }\right)^{2/3},
\end{equation}
where $\text{Vol}(K)$ denotes the volume of $K$ in $\mathbb{R}^3$.
\end{theorem}
This first Theorem A is a particular case of a more general result. Namely,
\begin{theorem}[Theorem B]
\label{mainresult}
Let $\Sigma$ be a properly embedded $[\varphi,\vec{e}_{3}]$-minimal surfaces in $\mathbb{R}^2\times ]a,b[$ with asymptotic flat ends and weight function $\varphi$ verifying \eqref{segcond}.  If $\Omega,\widetilde{\Omega}$ are the two connected components of $\Sigma\backslash\mathbb{R}^3$, then, for any $f\in\mathcal{H}^{1}_{loc}(\Sigma)$ such that $$\int_{\Sigma}e^{\varphi}f^2\, d\Sigma=1, \ \ \ \ \int_{\Sigma}e^{\varphi}f\, d\Sigma=0 \ \ \text{and} \ \ \int_{\Sigma}e^{\varphi}\vert\nabla f\vert^2\, d\Sigma<+\infty,$$ we get that the following inequality holds 
\begin{equation}
\label{mainineq}
\int_{\Sigma}\, e^{\varphi}\vert\nabla f\vert^2\, d\Sigma\geq C_{\mathcal{P}}(\Sigma)\, C_{\varphi}(\Sigma),
\end{equation}
where $C_{\mathcal{P}}(\Sigma)=\text{min}\{ C_{\mathcal{P}}(\Omega), C_{\mathcal{P}}(\widetilde{\Omega})\}$ and  $C_{\mathcal{P}}(\Omega), C_{\mathcal{P}}(\widetilde{\Omega})$ are the Poincare's constants of $\Omega$ and $\widetilde{\Omega}$, respectively.
\end{theorem}
On the other hand, from the well-known result of Fischer-Colbrie \cite{FCS} together with the Cohn'Vossen's inequality, we can guarantee the stability of $ \mathcal{L}^{\varphi}$ of the following family of surfaces and to get a control on the topology as follows
\begin{theorem}[Theorem C]
\label{secondresult}
Let $\Sigma$ be a properly embedded $[\varphi,\vec{e}_{3}]$-minimal surface in $\mathbb{R}^3$ with asymptotic flat ends, the weight function $\varphi$ verifying \eqref{segcond} with almost quadratic growth and  norm of the second fundamental form bounded \eqref{S}. If
$$2\int_{\Sigma}e^{\varphi}K\, f^2\, d\Sigma-\text{sup}_{\Sigma}H^2\geq C_{\mathcal{P}}(\Sigma)C_{\varphi}(\Sigma)-\text{inf}_{\Sigma}\ddot{\varphi},$$ 
for any $f\in\mathcal{H}^{1,\varphi}_{0}(\Sigma)$, then $\Sigma$ is stable. If in addition $\Sigma$ has finite total Gauss curvature , finite topology and 
$$\int_{\Sigma}(1+2e^{\varphi}f^2)K\, d\Sigma\geq 0, $$ then
$$2\pi\chi(\Sigma)+\lambda_{1}(\mathcal{L}^\varphi)\geq C_{\mathcal{P}}(\Sigma)\, C_{\varphi}(\Sigma)-\text{sup}_{\Sigma}H^2+\text{inf}_{\Sigma}\ddot{\varphi},$$
where $\chi(\Sigma)$ is the Euler's characteristic of $\Sigma$ and $\lambda_{1}(\mathcal{L}^{\varphi})$ is the first Neumann eingenvalue of $\mathcal{L}^{\varphi}$.
\end{theorem}

Finally, we would like point out that, under these constraints, we are no guaranteed the existence of stable examples. From \cite{CMZ}, if the Ilmanen's space is complete with $\ddot{\varphi}\leq-\epsilon<0$ for some real constant $\epsilon>0$, then there are no stable properly embedded $[\varphi,\vec{e}_{3}]$-minimal surfaces $\Sigma$ with $\mathcal{A}^{\varphi}(\Sigma)<+\infty$. 
\\

The paper is organized as follows: Section 2 defines the family of cutoff functions that we will use along of this work. Moreover, we find sufficient conditions to belong to space $\mathcal{H}^{1,\varphi}_{loc}(\mathbb{R}^3)$ and  $\mathcal{H}^{1,\varphi}_{loc}(\Sigma).$ In Section 3, we give the existence of minimum of a Dirichlet-Neumann problem which gives us an estimation for \eqref{mainineq}.  In fact, the lower bounded of Theorem \eqref{mainresult} together with the proofs of Theorem A and B will be obtained in Section 5. In Section 4, we gives a spectral theorem for the Jacobi operator $\mathcal{L}^{\varphi}$ and finally, in Section 6, we show the proof of Theorem C together with some topological consequences applying the Cohn'Vossen's inequality together with the work of B. White \cite{W} for complete surfaces with finite total curvature.

\section{First properties of $\mathcal{H}^{1,\varphi}_{loc}$. Cutoff functions.}
Throughout of this section, we will assume that  $\Sigma$ is a connected properly embedded $[\varphi,\vec{e}_{3}]$-minimal surface in $\mathbb{R}^3$ without boundary (possibly compact) and whose function $\varphi:\mathbb{R}\rightarrow [0,+\infty[$ satisfies \eqref{segcond} with compact support containing the interval $]a,b[$.

\begin{lemma}
 For any $\epsilon>0$, there exists $\{f_{n}\}_{n\in\mathbb{N}}$ a sequence of smooth functions which satisfy the following conditions
\begin{equation}
\label{eqf}
c_{n}-(f_{n}\circ\varphi)(n\, t)\rightarrow o(n) \ \ \text{and} \  \ \left(c_{n}-(f\circ\varphi)(n\, t)-(f\circ\varphi)'(n\, t)n\, t \right)\leq-\epsilon\, n,
\end{equation}
 for any $t>0$ and where $\{c_{n}\}_{n\in\mathbb{N}}$ is a sequence of real numbers such that $$c_{n}\leq- \epsilon n \ \ \text{ for any } \ \ n\in\mathbb{N}$$
Moreover, the function
$t\longrightarrow g_{n}(t)=e^{\left(c_{n}-(f_{n}\circ \varphi)(nt)\right)\frac{t}{n}}$
for any $t>0$ converges uniformly to $1$ in norm  $L^2$.
\end{lemma}
\begin{proof}
Fix any $\epsilon>0$ and consider $c_{n}<-\epsilon\, n$. Define $$f_{n}(t)=b_{n}+\text{log}\left(1+\frac{t^2}{n} \right),$$ where $b_{n}=c_{n}+\delta_{n}>0$ with $\delta_{n}\rightarrow 0$. It is clear that, $$c_{n}-(f_{n}\circ\varphi)(nt)=o(n).$$ Moreover, outside a compact set, we get the following inequalities
$$c_{n}-(f_{n}\circ\varphi)(nt)-(f_{n}\circ\varphi)'(nt)nt\leq c_{n}-\frac{2n^2t\varphi(nt)\dot{\varphi}(nt)}{n+\varphi(nt)^2}\leq c_{n}<-\epsilon n.$$
Finally, since $\varphi$ has compact support containing the interval $]a,b[$, we get
$$\vert g_{n}-1\vert^2\leq \bigg\vert \frac{\delta_{n}}{n}t+O\left(\frac{t^2}{n^2}\right)\bigg\vert^2\leq 2\frac{\delta_{n}^2}{n^2}t^2+2O\left(\frac{t^4}{n^4} \right)$$
and
$$\vert g_{n}'\vert^2\leq \bigg\vert e^{\frac{\delta_{n}}{n}t}\frac{\delta_{n}}{n}\bigg\vert^2\leq C\frac{\delta_{n}^2}{n^2},$$
for some constant $C>0$ and any $n$ large enough.
\end{proof}

On the other hand, consider a smooth function $\beta:[0,+\infty[\rightarrow[0,1]$ such that $\beta(t)=1$ for $t\in[0,1]$, $\beta'(t)\leq 0$ for $t\in [1,2]$ with $\beta(t)=0$ for $t>2$, $\vert\beta'\vert\leq 1$ and $$\beta'\left(\frac{\vert x\vert}{n}\right) \vert x\vert\in L^{2,\varphi}.$$

Now, for each $n\in\mathbb{N}$ we define $\eta_{n}:\mathbb{R}^3\rightarrow[0,1]$ given by
\begin{equation}
\label{etan}
\eta_{n}(x)=e^{\left(c_{n}-(f_{n}\circ\varphi)(n\, \vert x\vert^2) \right)\frac{\vert x\vert^2}{n}}\, \beta\left(\frac{\vert x\vert}{n^\rho} \right),
\end{equation}
where $\vert x\vert=\sqrt{\langle x,x\rangle}$. By definition, we can assume that $\eta_{n}\longrightarrow 1$ uniformly as  $n\rightarrow+\infty$ in norm $\mathcal{H}^{1,\varphi}(\mathbb{R}^3)$ taking $\rho$ small enough since $g_{n}$ converges uniformly to $1$ on norm of $H^1$. Furthermore, outside of origin, we get that
\begin{equation}
\label{grad1}
\overline{\nabla}\vert x\vert=\frac{x}{\vert x\vert} \ \ \text{ and } \ \ \langle\overline{\nabla}\vert x\vert,x\rangle=\vert x\vert,
\end{equation}
\begin{equation}
\label{grad2}
\nabla\vert x\vert=\frac{x^T}{\vert x\vert} \ \ \text{ and } \ \ \langle\nabla\vert x\vert,x^T\rangle=\frac{\vert x^T\vert^2}{\vert x\vert},
\end{equation}
where $x^T$ denotes the tangent projection of $x$. In particular, $\beta'  \langle\overline{\nabla}\vert x\vert, x\rangle\leq 0$ , $\beta'  \langle\nabla\vert x\vert, x^T\rangle\leq 0$ and
\begin{align*}
\langle\overline{\nabla}\eta_{n},x\rangle&\leq 2\eta_{n}\left(c_{n}-(f\circ\varphi)(n\vert x\vert^2)-(f\circ\varphi)'(n\vert x\vert^2)n\vert x\vert^2 \right)\frac{\vert x\vert^2}{n} , \\
\langle\nabla\eta_{n},x\rangle&\leq 2\eta_{n}\left(c_{n}-(f\circ\varphi)(n\vert x\vert^2)-(f\circ\varphi)'(n\vert x\vert^2)n\vert x\vert^2 \right)\frac{\vert x^T\vert^2}{n}.
\end{align*}
Thus, from \eqref{eqf}, we obtain the following result
\begin{lemma} Under assumptions as above, we get that
\begin{align}
&\langle\overline{\nabla}\eta_{n},x\rangle= -2\eta_{n}\epsilon\vert x\vert^2. \\
&\langle\nabla\eta_{n},x^T\rangle= -2\eta_{n}\epsilon\vert x^T\vert^2.
\end{align}
\end{lemma}
\begin{definition}
The set $\mathcal{H}^{1,\varphi}_{loc}(\mathbb{R}^3)$ denotes the set of all function $w\in H^{1}_{loc}(\mathbb{R}^3)$ such that
$$\int_{\mathbb{R}^3}\, e^{\varphi}\left( w^2+\vert\overline{\nabla}w\vert^2\right)\, dx<+\infty,$$
where $dx$ is the usual Lebesgue measure of $\mathbb{R}^3$. Analogously, $\mathcal{H}^{1,\varphi}_{loc}(\Sigma)$ denotes the set of all functions $f\in H^{1}_{loc}(\Sigma)$ such that $$\int_{\Sigma}\, e^{\varphi}\left( f^2+\vert\overline{\nabla}f\vert^2\right)d\Sigma<+\infty.$$
Finally, $\mathcal{H}^{1,\varphi}_{0}(\Sigma)\subset\mathcal{H}^{1,\varphi}_{loc}(\Sigma)$ is the set of all functions $f\in\mathcal{H}^{1,\varphi}_{loc}(\Sigma)$ such that
$$0=\int_{\Sigma}\, e^{\varphi}\, fd\Sigma=1-\int_{\Sigma}\, e^{\varphi}\, f^2d\Sigma$$
\end{definition}
\begin{lemma}
\label{propH1}
Assume that $w\in\mathcal{H}^{1,\varphi}_{loc}(\mathbb{R}^3)$ such that $\int_{\mathbb{R}^3}\, e^{\varphi}\, \vert\overline{\nabla}w\vert^2<+\infty.$ Then, there exists a constant $\overline{C}=\overline{C}(\dot{\varphi})>0$ (depending only of $\dot{\varphi}$) such that
\begin{equation}
\label{ineq1}
\int_{\mathbb{R}^3}e^{\varphi}\, w^2 \vert x\vert^2\, dx\leq\int_{B(0,1)}\, w^2\, dx+C\,\int_{\mathbb{R}^3}e^{\varphi}\, \vert\overline{\nabla}w\vert^2\, dx,
\end{equation}
where $B(0,1)$ is the unit ball of $\mathbb{R}^3$ centered in the origin. In particular, $w\in\mathcal{H}^{1,\varphi}_{loc}(\mathbb{R}^3)$.
\end{lemma}
\begin{proof}
Consider $w\in\mathcal{H}^{1}_{loc}(\mathbb{R}^3)$ such that $\int_{\mathbb{R}^3}e^{\varphi}\, \vert\overline{\nabla}w\vert^2<+\infty$ and take $\{\eta_{n}\}_{n\in\mathbb{N}}$ the sequence of functions defined in \eqref{etan}. From Lemma \eqref{propH1}, there exists $\varepsilon>0$ such that
\begin{align}
\overline{\text{div}}\left(e^{\varphi}\eta_{n}^2 w^2  x \right)&\leq-4\epsilon e^{\varphi}w^{2}\eta_{n}^{2}\vert x\vert^2+\eta_{n}^2\overline{\text{div}}(e^{\varphi}w^2 x)\label{ineq1} \\
&\leq-4\epsilon e^{\varphi}\eta_{n}^2w^2\vert x\vert^2+e^{\varphi}\eta_{n}^2w^2z\dot{\varphi}(z)\nonumber \\
&+e^{\varphi}\eta_{n}^2w^2\vert x\vert^2+3e^{\varphi}\eta_{n}^2w^2+e^{\varphi}\eta_{n}^2\vert\overline{\nabla}w\vert^2,\nonumber
\end{align}
where $\overline{\text{div}}$ is the divergence operator of $\mathbb{R}^3$. On the other hand,from \eqref{segcond}, we get  $z\dot{\varphi}(z)\leq 0$ and thus, applying the Divergence Theorem in $\mathbb{R}^3$ together with \eqref{ineq1}, we have
$$4(\epsilon-1)\int_{\mathbb{R}^3}e^{\varphi}\eta_{n}^2w^2\vert x\vert^2\, dx\leq \int_{B(0,1)}e^{\varphi}\eta_{n}^2w^2\, dx+\int_{\mathbb{R}^3}e^{\varphi}\eta_{n}^2\vert\overline{\nabla}w\vert^2\, dx.$$
Finally, taking $\epsilon>1$, we obtain the existence of a real constant $\overline{C}>0$ such that
$$\int_{\mathbb{R}^3}e^{\varphi}w^{2}\eta_{n}^{2}\vert x\vert^2\, dx\leq\int_{B(0,1)}\, w^{2}\eta_{n}^2\, dx+\overline{C}\, \int_{\mathbb{R}^3}e^{\varphi}\eta_{n}^{2}\vert\overline{\nabla}w\vert^2\, dx$$
and the proof follows taking limit as $n\rightarrow+\infty$. 
\end{proof}
\begin{lemma}
\label{propH2}
Assume that $f\in\mathcal{H}^{1}_{loc}(\Sigma)$ such that $\int_{\Sigma}e^{\varphi}\vert\nabla f\vert^2<+\infty$. Then, there exists a constant $C=C(\dot{\varphi})>0$ (depending only of $\dot{\varphi}$) such that
\begin{equation}
\label{ineq1}
\int_{\Sigma}e^{\varphi}\, f^2 \vert x^T\vert^2\, d\Sigma\leq\int_{\mathcal{B}(0,1)}\, f^2\, d\Sigma+C\,\int_{\Sigma}e^{\varphi}\, \vert\nabla f\vert^2\, d\Sigma,
\end{equation}
where $\mathcal{B}(0,1)=\Sigma\cap B(0,1)$. In particular, $f\in\mathcal{H}^{1,\varphi}_{loc}(\Sigma)$.
\end{lemma}
\begin{proof}
Consider $f\in\mathcal{H}^{1}_{loc}(\Sigma)$ such that $\int_{\Sigma}e^{\varphi}\, \vert\overline{\nabla}f\vert^2\, d\Sigma<+\infty$ and take $\{\eta_{n}\}_{n\in\mathbb{N}}$ the sequence of functions defined in \eqref{etan}. From Lemma \eqref{propH1}, there exists $\epsilon>0$ such that
\begin{align}
\text{div}\left(e^{\varphi}\eta_{n}^2 f^2  x^T \right)&\leq-4\epsilon e^{\varphi}f^{2}\eta_{n}^{2}\vert x^T\vert^2+\eta_{n}^2\overline{\text{div}}(e^{\varphi}f^2 x^T)\label{ineq2} \\
&\leq-4\epsilon e^{\varphi}\eta_{n}^2f^2\vert x^T\vert^2+(2+z\dot{\varphi}(z))e^{\varphi}\eta_{n}^2f^2 \nonumber \\
&+e^{\varphi}\eta_{n}^2f^2\vert x^T\vert^2+e^{\varphi}\eta_{n}^2\vert\nabla f\vert^2,\nonumber
\end{align}

where $\text{div}$ is the divergence operator of $\Sigma$. Since the set $\mathcal{B}(0,1)$ is compact due to Hopf-Rinow's completness and again, taking $\epsilon>1$,  we can apply a similar reasoning as Lemma \eqref{propH1} and \cite[Lemma 6]{BT}, to assure the existence of a constant $C=C(\dot{\varphi})>0$ such that
\begin{equation}
\label{lim1}
\int_{\Sigma}e^{\varphi}f^{2}\eta_{n}^{2}\vert x^T\vert^2\, d\Sigma\leq\int_{\mathcal{B}(0,1)}\, f^{2}\eta_{n}^2\, d\Sigma+C\, \int_{\Sigma}e^{\varphi}\eta_{n}^{2}\vert\nabla f\vert^2\, d\Sigma.
\end{equation}
In particular, since $\Sigma$ has flat ends and $\Sigma$ is immersed in $\mathbb{R}^2\times ]a,b[$, then $\langle X,N-\vec{e}_{3}\rangle\rightarrow 0$ for any divergent sequence of points of $\Sigma$. Thus, outside a compact set, there exists a constant $\delta>0$, depending of the height function $z$, such that $\vert x^T\vert^2\geq \vert x\vert^2-\delta$. Hence, from \eqref{lim1}, $f\in\mathcal{H}^{1,\varphi}_{loc}(\Sigma)$.
\end{proof}

\section{Minimum of a Dirichlet-Neumann problem.}
In this section, we adapt the arguments of \cite{BT} so that the work is self-contained. Let $\Sigma$ be a properly embedded $[\varphi,\vec{e}_{3}]$-minimal surface in $\mathbb{R}^3$ without boundary. Let $\Omega,\widetilde{\Omega}$ be the two connected components of $\mathbb{R}^3\backslash\Sigma$. Define $\mathcal{DN}_{loc}^{1,\varphi}$ the set of all non-constant functions $w\in\mathcal{H}_{loc}^{1,\varphi}(\mathbb{R}^3)$  satisfying 
\begin{equation}
\label{DNproblem}
\left.
\begin{array}{rr}
\overline{\Delta}^{\varphi} u=\overline{\Delta}^{\varphi}\widetilde{u}=0 \ \ \text{on} \ \ \Omega,\widetilde{\Omega} \ \, \ \ \text{respectively.} & \\
\alpha\left(\frac{\partial u}{\partial \nu} +\frac{\partial\widetilde{u}}{\partial\widetilde{\nu}} \right)=\Delta^{\varphi} f+\mu\, f \ \ \text{on} \ \ \Sigma. &
\end{array}
\right\}
\end{equation}
with $u=w\vert_{\Omega}$, $\widetilde{u}=w\vert_{\widetilde{\Omega}}$, $f=w\vert_{\Sigma}$, $\nu,\widetilde{\nu}$ are the normal vector pointing to $\Omega$, $\widetilde{\Omega}$ respectively, $\mu\in\mathbb{R}$ is a constant depending of $\alpha$.
\begin{definition}
We define $\mu$ to be the infimum of the functional
\begin{equation}
\label{functional}
\int_{\Sigma}e^{\varphi}\vert\nabla f\vert^2\, d\Sigma+\alpha\int_{\mathbb{R}^3}e^{\varphi}\vert\overline{\nabla}w\vert^2\, dx
\end{equation}
over all pair of functions $(f,w)\in\mathcal{H}^{1,\varphi}_{0}(\Sigma)\times\mathcal{H}^{1,\varphi}_{loc}(\mathbb{R}^3)$ with $f=w\vert_{\Sigma}$.
\end{definition}
\begin{proposition}
\label{existencemin}
There exists $(f,w)\in\mathcal{H}_{0}^{1,\varphi}(\Sigma)\times\mathcal{H}_{loc}^{1,\varphi}(\Sigma)$ with $f=w\vert_{\Sigma}$ such that
$$\int_{\Sigma}e^{\varphi}\vert\nabla f\vert^2\, d\Sigma+\alpha\int_{\mathbb{R}^3}e^{\varphi}\vert\overline{\nabla}w\vert^2\, dx=\mu.$$
\end{proposition}
\begin{proof}
By definition of minimum $\mu$, there exists a sequence of pairs of functions $\{(f_{n},w_{n})\}_{n\in\mathbb{N}}$ of $\mathcal{H}_{0.}^{1,\varphi}(\Sigma)\times\mathcal{H}_{loc}^{1,\varphi}(\mathbb{R}^3)$ with $w_{n}\vert_{\Sigma}=f_{n}$ satisfying that
\begin{equation}
\label{limt1}
\int_{\Sigma}e^{\varphi}\vert\nabla f_{n}\vert^2\, d\Sigma+\alpha\int_{\mathbb{R}^3}e^{\varphi}\vert\overline{\nabla}w_{n}\vert^2\, dx\longrightarrow\mu.
\end{equation}
Up to horizontal translation of $\Sigma$ , we can assume that $\Sigma\cap B(0,1)\neq\emptyset$.  Since each $f_{n}\in\mathcal{H}^{1,\varphi}_{loc}(\Sigma)$ together with Lemma \eqref{propH2}, we obtain the following uniformly bounds
$$\text{sup}_{n\in\mathbb{N}}\int_{\Sigma\cap B(0,1)}f_{n}^2\, d\Sigma<+\infty \ \ , \ \ \text{sup}_{n\in\mathbb{N}}\int_{\Sigma}e^{\varphi}f_{n}^2\vert x^T\vert\, d\Sigma<+\infty.$$
On the other hand, from the convergence of $w_{n}$, we get that
$$\text{sup}_{n\in\mathbb{N}}\int_{B(0,1)}\vert\overline{\nabla}w_{n}\vert^2\, dx<+\infty.$$
Thus, the Poincare's Lemma together with the Sobolev trace Theorem, there exists a sequence $\{c_{n}\}_{n\in\mathbb{N}}$ such that
$$\text{sup}_{n\in\mathbb{N}}\int_{B(0,1)}(w_{n}-c_{n})^2\, dx<+\infty  \  \ \text{and} \ \ \text{sup}_{n\in\mathbb{N}}\int_{\Sigma\cap B(0,1)}(f_{n}-c_{n})^2\, dx<+\infty$$
Consequently, from Lemma \eqref{propH1}, the following holds
$$\text{sup}_{n\in\mathbb{N}}\int_{\mathbb{R}^3}e^{\varphi}w_{n}\, dx<+\infty.$$
After passing to subsequence, $\{f_{n}\}_{n\in\mathbb{N}}$ converges to a $f$ weakly in $\mathcal{H}_{loc}^{1}(\Sigma)$ and $\{w_{n}\}$ converges to $w$ weakly in $\mathcal{H}^{1}_{loc}(\mathbb{R}^3)$. Now, consider the cut-off functions $\eta_{m}$ defined in \eqref{etan}. For each $m$ large enough, we can choose an arbitrary $1>\epsilon>0$ such that $\eta_{m}^2>(1-\epsilon)^2\beta^2 (\vert x\vert/m)$. Since each $f_{n}\in\mathcal{H}_{0}^{1,\varphi}$, then
\begin{align}
& \bigg\vert 1-\int_{\Sigma}e^{\varphi}\eta_{m}^2 f_{n}^2\, d\Sigma\bigg\vert=\bigg\vert\int_{\Sigma}e^{\varphi}(1-\eta_{m}^2)f_{n}^2\, d\Sigma \bigg\vert\leq\epsilon(2-\epsilon) +\frac{1}{m}\int_{\Sigma}e^{\varphi}\vert\beta'\vert \vert x\vert f_{n}^2\, d\Sigma \label{flimit} \\
& \bigg\vert\int_{\Sigma}e^{\varphi}\eta_{m} f_{n}\, d\Sigma\bigg\vert\leq \bigg\vert\int_{\Sigma}e^{\varphi}(1-\eta_{m})f_{n}\, d\Sigma \bigg\vert\leq \left(\int_{\Sigma}e^{\varphi}(1-\eta_{m})^2\right)^{1/2} \left(\int_{\Sigma}e^{\varphi}f_{n}^2\right)^{1/2} .\label{slimit}
\end{align}
where in the last inequality we used the H\"older inequality. Now, from Lemma \eqref{propH2} and taking into account that $\beta'\vert x\vert\in L^{2,\varphi}(\Sigma)$, we can take limits as $n,m\rightarrow+\infty$ in \eqref{flimit} and \eqref{slimit}, to proving that $f\in\mathcal{H}^{1,\varphi}_{0}(\Sigma)$. Finally, using the lower-semicontinuty, we obtain
$$\int_{\Sigma}e^{\varphi}\vert\nabla f\vert^2\, d\Sigma+\alpha\int_{\mathbb{R}^3}e^{\varphi}\vert\overline{\nabla} w\vert^2\, dx\leq\text{liminf}\left(\int_{\Sigma}e^{\varphi}\vert\nabla f_{n}\vert^2\, d\Sigma+\alpha\int_{\mathbb{R}^3}e^{\varphi}\vert\overline{\nabla} w_{n}\vert^2\right)\, dx \leq\mu.$$
In particular,
$$\int_{\Sigma}e^{\varphi}\vert\nabla f\vert^2\, d\Sigma<+\infty \ \ \text{and} \ \ \int_{\mathbb{R}^3}e^{\varphi}\vert\overline{\nabla} w\vert^2\, dx<+\infty$$
and
$$\int_{\Sigma}e^{\varphi}\vert\nabla f\vert^2\, d\Sigma+\alpha\int_{\mathbb{R}^3}e^{\varphi}\vert\overline{\nabla} w\vert^2\, dx= \mu.$$
\end{proof}
\begin{corollary}
\label{equtest}
Suppose that $(f,w)$ is the minimazer of the previous Proposition \eqref{existencemin}. Then, for any function $(\phi,\psi)\in\mathcal{H}^{1,\varphi}_{0}(\Sigma)\times\mathcal{H}^{1,\varphi}_{loc}(\mathbb{R}^3)$  such that $\psi\vert_{\Sigma}=\phi$, the following equality holds,
$$\int_{\Sigma}e^{\varphi}\langle\nabla f,\nabla\phi \rangle\, d\Sigma+\alpha\int_{\mathbb{R}^3}e^{\varphi}\langle\overline{\nabla}w,\overline{\nabla}\psi\rangle\, dx= \mu\int_{\Sigma}e^{\varphi}f\phi\, d\Sigma,$$
\end{corollary}
\begin{proof}
Consider any function test $\phi$ with compact support in $\Sigma\cap B(0,r_{0})$ for some $r_{0}>0$ large enough and define
$$a_{r_n}=-\dfrac{\int_{\Sigma\cap B(0,r_{n})}e^{\varphi}\phi\, d\Sigma}{\int_{\Sigma\cap B(0,r_{n})}e^{\varphi}\, d\Sigma}.$$
Notice that $a_{r_{n}}=a_{r_{0}}$ for any $r_{n}>r_{0}$. On the other hand, define $\widetilde{f}_{s}=f+s(\phi+a_{r_n})$ and $\widetilde{w}_{s}=w+s(\psi+a_{r_n})$. It is clear that $\widetilde{w}_{s}\vert_{\Sigma}=\widetilde{f}_{s}$ and by straighforward computation, we get that
$$0=\int_{\Sigma\cap B(0,r_{n})}e^{\varphi}(\widetilde{f}_{s}-f)\, d\Sigma=\int_{\Sigma}e^{\varphi}(\widetilde{f}_{s}-f)\chi_{r_{n}}\, d\Sigma, $$
for any $r_{n}>r_{0}$ and where $\chi_{r_{n}}$ is the characteristic function on $B(0,r_{n})$. Notice that, if $g_{n}=(\widetilde{f}_{s}-f)\chi_{r_{n}}$, then $g_{n}$ converges uniformly to $\widetilde{f}_{s}-f$ in norm $L^{2,\varphi}(\Sigma)$ as $n\rightarrow+\infty$. Thus, we can take limit to proving that
$$\int_{\Sigma}e^{\varphi}\widetilde{f}_{s}\, d\Sigma=0.$$
In particular, up to constant, we get that $\widetilde{f}_{s}\in\mathcal{H}^{1,\varphi}_{0}$ and the rest of the proof follows as \cite{BT} together with a density argument in $\mathcal{H}^{1,\varphi}_{0}(\Sigma)$.
\end{proof}
\begin{corollary}
\label{smooth}
Suppose that $(f,w)$ is the minimazer of the previous Proposition \eqref{existencemin}. Then, $w\in\mathcal{DN}^{1,\varphi}_{loc}(\mathbb{R}^3)$ and $f\in\mathcal{C}^{\infty}(\Sigma)$. Moreover, the functions $u=w\vert_{\Omega}$ and $\widetilde{u}=w\vert_{\widetilde{\Omega}}$ are smooth up to boundary of $\Omega$ and $\widetilde{\Omega}$, respectively.
\end{corollary}
\begin{proof}
We only need to prove the regularity of $f$ since from the Corollary \eqref{equtest}, we get that $w$ is non-constant function lying in $\mathcal{DN}^{1,\varphi}$. The rest of the proof is an adaptation of \cite[Lemma 11]{BT} and \cite[Lemma 12]{BT}. Fix an integer $m\geq 1$, there exists $r_{0}$ small enough where $f\in\mathcal{H}^{m}(\Sigma\cap B(0,r))$, $u\in\mathcal{H}^{m}(\Omega\cap B(0,r))$, $\widetilde{u}\in\mathcal{H}^{m}(\widetilde{\Omega}\cap B(0,r))$, for any $r\in]0,r_{0}[$. Moreover, notice that, $\langle\nabla u,\nu\rangle, \langle\nabla \widetilde{u},\widetilde{\nu}\rangle\in\mathcal{H}^{m-1}(\Sigma\cap B(0,r_{0}))$. Then, from the Elliptic regularity theory implies that $f\in\mathcal{H}^{m+1}(\Sigma\cap B(0,r))$ for any $r]0,r_{0}[$. Hence, an analogous argument shows that $u\in\mathcal{H}^{m+1}(\Omega\cap B(0,r))$ and $\widetilde{u}\in\mathcal{H}^{m+1}(\widetilde{\Omega}\cap B(0,r))$. These implies that $f\in\mathcal{H}^{m}(\Sigma\cap B(0,r))$, $\widetilde{u}\in\mathcal{H}^{m}(\widetilde{\Omega}\cap B(0,r))$ and $\widetilde{u}\in\mathcal{H}^{m}(\widetilde{\Omega}\cap B(0,r))$ for any integer $m\geq 1$. This proves that $f$ is smooth and the functions $u=w\vert_{\Omega}$ and $\widetilde{u}=w\vert_{\widetilde{\Omega}}$ are smooth up to boundary of $\Omega$ and $\widetilde{\Omega}$, respectively.
\end{proof}

\section{Spectral Theorem. Jacobi operator.}
Let $\Sigma$ be a properly embedded $[\varphi,\vec{e}_{3}]$-minimal surface in $\mathbb{R}^3$ where $\varphi$ is a smooth function with almost quadratic growth. Moreover, we also assume that the norm of the second fundametal form $\vert S\vert^{2}$ of $\Sigma$ is bounded.
\\

Since $\Sigma$ is a critical point of the weighted area $\mathcal{A}^{\varphi}$, we consider the following bilinear simetric form $\mathcal{Q}:\mathcal{H}^{1,\varphi}_{loc}(\Sigma)\times\mathcal{H}^{1,\varphi}_{loc}(\Sigma)\rightarrow\mathbb{R}$ given by
\begin{equation}
\label{defQ}
\mathcal{Q}(u,\phi)=\int_{\Sigma}e^{\varphi}\left(\langle\nabla u,\nabla \phi\rangle-qu\phi\right)\, d\Sigma,
\end{equation}
where $q=\vert S\vert^2-\ddot{\varphi}\eta^2\in\mathcal{C}^{\infty}(\Sigma)\cap L^{\infty}(\Sigma)$ which is associated with the Jacobi operator of stability $\mathcal{L}:\mathcal{H}^{1,\varphi}_{loc}(\Sigma)\rightarrow\mathcal{H}^{1,\varphi}_{loc}(\Sigma)$ defined by
\begin{equation}
\label{defL}
\mathcal{L}=\Delta^{\varphi}+q.
\end{equation}
In this context, the weak solutions $u$ of the problem $\mathcal{L}u=-\lambda u$ associtated to eigenvalue $\lambda$ can be characterized by functions $u\in\mathcal{H}^{1,\varphi}_{loc}(\Sigma)$ satisfying
$$\mathcal{Q}(u,\phi)=\lambda\int_{\Sigma}u\phi\, d\Sigma,$$
for any test function $\phi\in\mathcal{H}^{1,\varphi}_{0}(\Sigma)$ with compact support. Moreover, the first eigenvalue is given by
\begin{equation}
\label{defl1}
\lambda_{1}(\mathcal{L}^{\varphi})=\text{inf}_{\mathcal{H}^{1,\varphi}_{loc}(\Sigma)}\frac{\int_{\Sigma}e^{\varphi}\left(\vert\nabla u\vert^2-qu^2 \right)\, d\Sigma}{\int_{\Sigma}e^{\varphi}u^2\, d\Sigma}.
\end{equation}
Notice that, the quotient \eqref{defl1} is invariant by scaling. Thus, we can take, without loss of generality, the infimum in $\mathcal{H}^{1,\varphi}_{0}(\Sigma)$.
\\

The following Courant-type nodal domain result was proved by R. Chen, J. Mao and C. Wu \cite[Theorem 1.2]{CMW} which stablish the following structure
\begin{theorem}
Assume that $\Omega$ is a regular domain on a given Riemannian manifold. For the Dirichlet eigenvalue problem for $\Delta^{\varphi}$, its eigenvalues consist of a non-decreasing sequence $0=\lambda_{0}<\lambda_{1}\leq\lambda_{2}\leq\cdots\leq\cdot+\infty$. Denote by $f_{i}$ an eigenfunction of the $i$-th eigenvalue $\lambda_{i}$ and $\{f_{1},f_{2},f_{3},\cdots\}$ forms a complete orthogonal basis of $L^2(\Omega)$. Then, for each $k=1,2,3\cdots$, the number of nodal domains of $f_{k}$ is less than or equal to $k$.
\end{theorem}

Second, we characterize variationally the first eigenvalue  $\lambda_{1}(\mathcal{L}^{\varphi})$ for the Jacobi operator $\mathcal{L}^{\varphi}$ on the space of functions $\mathcal{H}^{1,\varphi}_{0}(\Sigma)$.

\begin{theorem} Let $\Sigma$ be a properly embedded $[\varphi,\vec{e}_{3}]$-minimal surface in $\mathbb{R}^3$ with bounded $\vert S\vert^2$ and function $\varphi$ satisfying \eqref{segcond}. Then,
\begin{equation}
\label{deffl1}
\lambda_{1}(\mathcal{L}^{\varphi})=\text{inf}_{\mathcal{H}^{1,\varphi}_{0}(\Sigma)}\bigg\{\int_{\Sigma}e^{\varphi}\left(\vert\nabla u\vert^2-qu^2 \right)\, d\Sigma\bigg\}=\text{inf}_{\mathcal{H}^{1,\varphi}_{0}(\Sigma)}\mathcal{Q}(u,u).
\end{equation}
and there exists $f_{1}\in\mathcal{H}^{1,\varphi}_{0}(\Sigma)$ such that $\mathcal{Q}(f_{1},f_{1})=\lambda_{1}(\mathcal{L}^{\varphi})$. Moreover, all eigenfunctions associated to $\lambda_{1}(\Sigma)$ can be characterized variationally by
$$\mathcal{Q}(u,u)\geq \lambda_{1}(\mathcal{L}^{\varphi}) \ \ \text{for any} \ \ u\in\mathcal{H}^{1,\varphi}_{0}(\Sigma)$$
with equality if and only if $u$ is eigenfunction of $\lambda_{1}(\mathcal{L}^{\varphi})$.
\end{theorem}
\begin{proof}
Firstly, the existence of this infimum is guaranteed since $q\in\mathcal{L}^{\infty}(\Sigma)$ and 
$$\mathcal{Q}(u,u)\geq-\int_{\Sigma}e^{\varphi} qu^2\, d\Sigma\geq -\text{sup}_{\Sigma}q\int_{\Sigma}e^{\varphi}u^2\, d\Sigma\geq-\text{sup}_{\Sigma}q.$$
By definition, there exists a sequence $\{u_{n}\}_{n\in\mathbb{N}}$ of functions in $\mathcal{H}^{1,\varphi}_{0}(\Sigma)$ such that $$\mathcal{Q}(u_{n},u_{n})\rightarrow\lambda_{1}(\mathcal{L}^{\varphi}).$$
Moreover, the sequence $\{\mathcal{Q}(u_{n},u_{n})\}_{n\in\mathbb{N}}$ is upper bounded and thus,  $\{u_{n}\}_{n\in\mathbb{N}}$ is bounded in norm of $\mathcal{H}^{1,\varphi}_{loc}(\Sigma)$. Hence, up to subsequence, $\{u_{n}\}_{n\in\mathbb{N}}$ converges weakly to $f_{1}$ in $\mathcal{H}^{1,\varphi}_{loc}(\Sigma)$ and $\{u_{n}\}_{n\in\mathbb{N}}$ converges to $f_{1}$ in $L^{2,\varphi}(\Sigma)$. Then, $f_{1}\in\mathcal{H}^{1,\varphi}_{0}(\Sigma)$ since applying the H\"older inequality,
$$ \bigg\vert\int_{\Sigma}e^{\varphi}\eta_{m} u_{n}\, d\Sigma\bigg\vert\leq \bigg\vert\int_{\Sigma}e^{\varphi}(1-\eta_{m})u_{n}\, d\Sigma \bigg\vert\leq \left(\int_{\Sigma}e^{\varphi}(1-\eta_{m})^2\right)^{1/2}\left(\int_{\Sigma}e^{\varphi}u_{n}^2\right)^{1/2}$$
 for any $n\in\mathbb{N}$. Finally, from the lower-semicontinuity we obtain
$$\lambda_{1}(\mathcal{L}^{\varphi})\leq\mathcal{Q}(f_{1},f_{1})\leq \text{liminf}_{n\rightarrow+\infty}\mathcal{Q}(u_{n},u_{n})=\lambda_{1}(\mathcal{L}^{\varphi}).$$
Now, the eigenfunctions associated to $\lambda_{1}(\mathcal{L}^{\varphi})$ can be characterized variationally as follows: take $u\in\mathcal{H}^{1,\varphi}_{0}(\Sigma)$ such that $$\mathcal{Q}(u,\phi)=\lambda_{1}(\mathcal{L}^{\varphi})\int_{\Sigma}e^{\varphi}u\phi\, d\Sigma,$$ for any test function $\phi\in\mathcal{H}^{1,\varphi}_{0}(\Sigma)$ with compact support. Consider the following polynomial 
$$h(t)=\mathcal{Q}(u+t\phi,u+t\phi)-\lambda_{1}(\mathcal{L}^{\varphi})\int_{\Sigma}e^{\varphi}(u+t\phi)^2\, d\Sigma \ \ \text{for} \ \ t\in\mathbb{R}.$$
Notice that $h(t)\geq 0$ with $h(0)=0$. Thus, $h'(0)=0$ which implies that
$$\mathcal{Q}(u,\phi)=\lambda_{1}(\mathcal{L}^{\varphi})\int_{\Sigma}e^{\varphi}u\phi\, d\Sigma.$$
\end{proof}
\section{Estimation of the minimum $\mu$}
Take $(f,w)\in\mathcal{H}^{1,\varphi}_{0}(\Sigma)\times\mathcal{H}^{1,\varphi}_{loc}(\mathbb{R}^3)$ given by the Proposition \ref{existencemin} with fixed any $0<\alpha\leq 1$. Now, consider $\Omega, \widetilde{\Omega}$ be the two connected components of $\mathbb{R}^3\backslash\Sigma$ and the smooth function $u=w\vert_{\Omega}$ and $\widetilde{u}=w\vert_{\widetilde{\Omega}}$. Notice that, from Corollary \ref{smooth}, $w\in\mathcal{DN}^{1,\varphi}_{loc}$ and $f\in\mathcal{C}^{\infty}(\Sigma)$. Moreover, the functions $u=w\vert_{\Omega}$ and $\widetilde{u}=w\vert_{\widetilde{\Omega}}$ are smooth up to boundary of $\Omega$ and $\widetilde{\Omega}$, respectively. Thus, by a straighforward computation together with the well-known Bouchner's formula, we get that
\begin{equation}
\label{eqp1}
\frac{1}{2}\text{div}_{\Omega}\left(e^{\varphi}\overline{\nabla}\vert\overline{\nabla}u \vert^2 \right)=e^{\varphi}\left(\vert\overline{D}^2 u\vert^2-\overline{D}^2\varphi(\overline{\nabla} u,\overline{\nabla} u)\right),
\end{equation}
where $\overline{D}$ denotes the Euclidean Hessian in $\mathbb{R}^3$. Consequently, multiplying the equation \eqref{eqp1} by $\eta_{n}^2$ defined in \eqref{etan}, we obtain applying the Divergence Theorem that
\begin{align}
\frac{1}{2}\int_{\Sigma}e^{\varphi}\eta_{n}^2\langle\overline{\nabla}\vert\overline{\nabla} u \vert^2 , \nu \rangle\, d\Sigma=&\int_{\Omega}e^{\varphi}\eta_{n}^2\vert \overline{D}^2 u\vert^2\, dx-\int_{\Omega}e^{\varphi}\eta_{n}^2\overline{D}^2\varphi(\overline{\nabla} u,\overline{\nabla} u)\, dx \nonumber
 \\&+\int_{\Omega}e^{\varphi}\eta_{n}\langle\overline{\nabla}\eta_{n},\overline{\nabla}\vert\overline{\nabla} u\vert^2 \rangle\, dx. \label{eqp2}
\end{align}
Now, we compute the first term of \eqref{eqp2}. Firstly,
$$e^{\varphi}\langle\overline{\nabla}\vert\overline{\nabla} u\vert^2,\nu\rangle=e^{\varphi}\left(\overline{D}^2 u(\nabla f,\nu)+\overline{D}^2 u(\nu,\nu)\langle\overline{\nabla}u,\nu\rangle \right).$$
with  $\overline{D}^2 u(\nu,\nu)=\overline{\Delta}u-\Delta u-H\langle\overline{\nabla} u,\nu\rangle=-\Delta^{\varphi} f$. Moreover,
\begin{align*}
e^{\varphi}\overline{D}^2 u(\nabla f,\nu)=&e^{\varphi}\left(\langle \nabla f, \nabla \langle\overline{\nabla} u,\nu \rangle \rangle-\mathcal{S}(\nabla f,\nabla f)\right) \\
=& \text{div}_{\Sigma}(e^{\varphi}\langle\overline{\nabla} u,\nu \rangle\nabla f)-e^{\varphi}\langle\overline{\nabla} u,\nu\rangle\Delta^{\varphi} f-e^{\varphi}\mathcal{S}(\nabla f,\nabla f),
\end{align*}
where $\mathcal{S}$ stands the second fundamental form of $\Sigma$ with respect to $\nu$. Thus,
\begin{equation}
\label{eqp3}
e^{\varphi}\langle\overline{\nabla}\vert\overline{\nabla} u\vert^2,\nu\rangle=-2e^{\varphi}\langle\overline{\nabla} u,\nu\rangle\Delta^{\varphi} f+\text{div}_{\Sigma}\left(e^{\varphi}\langle\overline{\nabla} u,\nu \rangle\nabla f \right)-e^{\varphi}\mathcal{S}(\nabla f,\nabla f).
\end{equation}
Multiplying the previous equation \eqref{eqp3} by $\eta_{n}^2$ and applying the Divergence Theorem, we obtain that
\begin{align}
\frac{1}{2}\int_{\Sigma}e^{\varphi}\eta_{n}^2\langle\overline{\nabla}\vert\overline{\nabla} u\vert^2,\nu \rangle\, d\Sigma =& -\int_{\Sigma}e^{\varphi}\eta_{n}\langle\overline{\nabla} u,\nu \rangle \langle \nabla f, \nabla \eta_{n} \rangle\, d\Sigma-\int_{\Sigma}e^{\varphi}\eta_{n}^2\langle\overline{\nabla}u,\nu\rangle\Delta^{\varphi}f\, d\Sigma \nonumber \\
&-\frac{1}{2}\int_{\Sigma}e^{\varphi}\eta_{n}^2\mathcal{S}(\nabla f,\nabla f)\, d\Sigma. \label{eqp4}
\end{align}
Plug-in the equation \eqref{eqp4} into \eqref{eqp2}, the following equation holds
\begin{align}
&-\int_{\Sigma}e^{\varphi}\eta_{n}\langle \overline{\nabla} u,\nu\rangle\langle\nabla f,\nabla \eta_{n} \rangle\, d\Sigma-\int_{\Sigma}e^{\varphi}\eta_{n}^2\langle\overline{\nabla}u,\nu\rangle\Delta^{\varphi}f\, d\Sigma-\frac{1}{2}\int_{\Sigma}e^{\varphi}\eta_{n}^2\mathcal{S}(\nabla f,\nabla f)\, d\Sigma \label{eqp5} \\
&=\int_{\Omega}e^{\varphi}\eta_{n}^2\vert \overline{D}^2 u\vert^2\, dx-\int_{\Omega}e^{\varphi}\eta_{n}^2\overline{D}^2\varphi(\overline{\nabla} u,\overline{\nabla} u)\, dx+\int_{\Omega}e^{\varphi}\eta_{n}\langle\overline{\nabla}\eta_{n},\overline{\nabla}\vert\overline{\nabla} u\vert^2 \rangle\, dx.  \nonumber
\end{align}
As the same as before, we also obtain the following formula
\begin{align}
&-\int_{\Sigma}e^{\varphi}\eta_{n}\langle \overline{\nabla} \widetilde{u},\widetilde{\nu}\rangle\langle\nabla f,\nabla \eta_{n} \rangle-\int_{\Sigma}e^{\varphi}\eta_{n}^2\langle\overline{\nabla}\widetilde{u},\widetilde{\nu}\rangle\Delta^{\varphi}f-\frac{1}{2}\int_{\Sigma}e^{\varphi}\eta_{n}^2\widetilde{\mathcal{S}}(\nabla f,\nabla f) \label{eqp6} \\
&=\int_{\widetilde{\Omega}}e^{\varphi}\eta_{n}^2\vert \overline{D}^2 \widetilde{u}\vert^2-\int_{\widetilde{\Omega}}e^{\varphi}\eta_{n}^2\overline{D}^2\varphi(\overline{\nabla} \widetilde{u},\overline{\nabla} \widetilde{u})+\int_{\widetilde{\Omega}}e^{\varphi}\eta_{n}\langle\overline{\nabla}\eta_{n},\overline{\nabla}\vert\overline{\nabla} \widetilde{u}\vert^2 \rangle, \nonumber
\end{align}
where  $\widetilde{\mathcal{S}}$ is the second fundamental form of $\Sigma$ with respect to $\widetilde{\nu}$ given by $\widetilde{\mathcal{S}}=-\mathcal{S}$.
Next, from \eqref{DNproblem}, if we  add the equations \eqref{eqp5} and \eqref{eqp6}, then
\begin{align}
&-\int_{\Sigma}e^{\varphi}\eta_{n}\langle\nabla f,\nabla \eta_{n}\rangle\left(\langle\overline{\nabla} u,\nu\rangle+\langle\overline{\nabla}\widetilde{u},\widetilde{\nu} \rangle\right)\, d\Sigma+\mu\int_{\Sigma}e^{\varphi}\eta_{n}^2 f\left(\langle\overline{\nabla} u,\nu\rangle+\langle\overline{\nabla}\widetilde{u},\widetilde{\nu}\rangle \right)\, d\Sigma \label{eqp7} \\
&-\alpha\int_{\Sigma}e^{\varphi}\eta_{n}^2\left(\langle\overline{\nabla} u,\nu\rangle+\langle\overline{\nabla}\widetilde{u},\widetilde{\nu}\rangle \right)^2\, d\Sigma =\int_{\Omega}e^{\varphi}\eta_{n}^2\vert\overline{D}^2 u \vert^2\, dx+\int_{\widetilde{\Omega}}e^{\varphi}\eta_{n}^2\vert\overline{D}^2 \widetilde{u} \vert^2\, dx \nonumber \\
&-\int_{\Omega}e^{\varphi}\eta_{n}^2\overline{D}^2\varphi(\overline{\nabla} u,\overline{\nabla} u)\, dx -\int_{\widetilde{\Omega}}e^{\varphi}\eta_{n}^2\overline{D}^2\varphi(\overline{\nabla} \widetilde{u},\overline{\nabla} \widetilde{u})\, dx+\int_{\Omega}e^{\varphi}\eta_{n}\langle \overline{\nabla}\eta_{n},\overline{\nabla}\vert\overline{\nabla} u\vert^2\rangle\, dx \nonumber \\
&\hspace{ 3.7 cm} +\int_{\widetilde{\Omega}}e^{\varphi}\eta_{n}\langle \overline{\nabla}\eta_{n},\overline{\nabla}\vert\overline{\nabla} \widetilde{u}\vert^2\rangle\, dx. \nonumber
\end{align}
Since,
\begin{align*}
&\text{div}_{\Omega}\left(e^{\varphi}\eta_{n}^2 u \overline{\nabla} u \right)=e^{\varphi}\eta_{n}^2 \vert\overline{\nabla} u\vert^2+2e^{\varphi}\eta_{n} u\langle\overline{\nabla} u,\overline{\nabla}\eta_{n}\rangle \\
&\text{div}_{\widetilde{\Omega}}\left(e^{\varphi}\eta_{n}^2 \widetilde{u} \overline{\nabla} \widetilde{u} \right)=e^{\varphi}\eta_{n}^2 \vert\overline{\nabla} \widetilde{u}\vert^2+2e^{\varphi}\eta_{n} \widetilde{u}\langle\overline{\nabla} \widetilde{u},\overline{\nabla}\eta_{n}\rangle
\end{align*}
we prove the following formula
\begin{align}
&\mu\int_{\Omega}e^{\varphi}\eta_{n}^2\vert\overline{\nabla} u\vert^2\, dx+2\int_{\Omega}e^{\varphi}\eta_{n} u \langle\overline{\nabla} u,\overline{\nabla}\eta_{n} \rangle\, dx+\mu\int_{\widetilde{\Omega}}e^{\varphi}\eta_{n}^2\vert\overline{\nabla} \widetilde{u}\vert^2\, dx+2\int_{\widetilde{\Omega}}e^{\varphi}\eta_{n} \widetilde{u} \langle\overline{\nabla} \widetilde{u},\overline{\nabla}\eta_{n} \rangle\, dx \label{eqp8} \\
&-\int_{\Sigma}e^{\varphi}\eta_{n}\langle\nabla f,\nabla \eta_{n}\rangle\left(\langle\overline{\nabla} u,\nu\rangle+\langle\overline{\nabla}\widetilde{u},\widetilde{\nu} \rangle\right)\, d\Sigma-\alpha\int_{\Sigma}e^{\varphi}\eta_{n}^2\left(\langle\overline{\nabla} u,\nu\rangle+\langle\overline{\nabla}\widetilde{u},\widetilde{\nu}\rangle \right)^2\, d\Sigma \nonumber \\
&=\int_{\Omega}e^{\varphi}\eta_{n}^2\vert\overline{D}^2 u \vert^2\, dx+\int_{\widetilde{\Omega}}e^{\varphi}\eta_{n}^2\vert\overline{D}^2 \widetilde{u} \vert^2\, dx-\int_{\Omega}e^{\varphi}\eta_{n}^2\overline{D}^2\varphi(\overline{\nabla} u,\overline{\nabla} u)\, dx -\int_{\widetilde{\Omega}}e^{\varphi}\eta_{n}^2\overline{D}^2\varphi(\overline{\nabla} \widetilde{u},\overline{\nabla} \widetilde{u})\, dx \nonumber \\
&\hspace{ 2.5 cm} +\int_{\Omega}e^{\varphi}\eta_{n}\langle \overline{\nabla}\eta_{n},\overline{\nabla}\vert\overline{\nabla} u\vert^2\rangle\, dx+\int_{\widetilde{\Omega}}e^{\varphi}\eta_{n}\langle \overline{\nabla}\eta_{n},\overline{\nabla}\vert\overline{\nabla} \widetilde{u}\vert^2\rangle\, dx. \nonumber
\end{align}
Furthermore, notice that 
\begin{align*}
&\eta_{n}^2\vert \overline{D}^2 u\vert^2+\eta_{n}\langle\overline{\nabla}\eta_{n},\overline{\nabla}\vert\overline{\nabla} u\vert^2\rangle\geq-\vert \overline{\nabla}\eta_{n}\vert^2\vert\overline{\nabla} u\vert^2  \text{ on } \Omega, \\
&\eta_{n}^2\vert \overline{D}^2 \widetilde{u}\vert^2+\eta_{n}\langle\overline{\nabla}\eta_{n},\overline{\nabla}\vert\overline{\nabla} \widetilde{u}\vert^2\rangle\geq-\vert \overline{\nabla}\eta_{n}\vert^2\vert\overline{\nabla} \widetilde{u}\vert^2 \text{ on } \widetilde{\Omega} \\
&\eta_{n}\langle \nabla f,\nabla \eta_{n}\rangle \left(\langle\overline{\nabla} u,\nu \rangle+\langle\overline{\nabla} \widetilde{u},\widetilde{\nu} \rangle \right)\geq -\vert\nabla f\vert \vert\nabla \eta_{n}\vert \vert\vert\overline{\nabla} u\vert+\vert\overline{\nabla}\widetilde{u} \vert \vert \text{ on } \Sigma.
\end{align*}
Since $f\in\mathcal{H}^{1,\varphi}_{0}(\Sigma)$ and $w\in\mathcal{DN}^{1,\varphi}_{loc}(\mathbb{R}^3)$, then
$$\int_{\Sigma}e^{\varphi}\vert\nabla f\vert^2\, d\Sigma<\infty \ \ \text{and} \ \ \int_{\mathbb{R}^3}e^{\varphi}(w^2+\vert\overline{\nabla}w\vert^2)\, dx<\infty.$$
As a consequence, we can take limit in the expresion \eqref{eqp8} to obtain the following inequality
\begin{align}
\mu\int_{\Omega}e^{\varphi}\vert\overline{\nabla} u\vert^2\, dx+&\int_{\Omega}e^{\varphi}\overline{D}^2\varphi(\overline{\nabla} u,\overline{\nabla} u)\, dx \nonumber \\
+&\mu\int_{\widetilde{\Omega}}e^{\varphi}\vert\overline{\nabla} \widetilde{u}\vert^2\, dx+\int_{\widetilde{\Omega}}e^{\varphi}\overline{D}^2\varphi(\overline{\nabla} \widetilde{u},\overline{\nabla}\widetilde{u})\, dx\geq 0 \label{eqfinal}.
\end{align}
In particular, from \eqref{eqfinal} together with \eqref{DNproblem} and \eqref{segcond}, we have
\begin{align}
\mu\int_{\Omega}e^{\varphi}\vert\overline{\nabla} u\vert^2\, dx-C_{\varphi}(\Sigma)\int_{\Omega}\overline{\Delta}^2\, u\, dx +\mu\int_{\widetilde{\Omega}}e^{\varphi}\vert\overline{\nabla} \widetilde{u}\vert^2\, dx-C_{\varphi}(\Sigma)\int_{\widetilde{\Omega}}\overline{\Delta}^2\, \widetilde{u}\, dx\geq 0 . \label{eqfinal2}
\end{align}
Finally, from the Poincare's inequality , there exists a constat $C_\mathcal{P}(\Sigma)=\text{min}\{ C_{\mathcal{P}}(\Omega), C_{\mathcal{P}}(\widetilde{\Omega})\}$ such that
$$\mu\geq C_{\mathcal{P}}(\Sigma) C_{\varphi}(\Sigma),$$
 where $ C_{\mathcal{P}}(\Omega), C_{\mathcal{P}}(\widetilde{\Omega})$ are the Poincare's constants of $\Omega$ and $\widetilde{\Omega}$, respectively.
\begin{proposition}
For any $f\in\mathcal{H}^{1,\varphi}_{0}(\Sigma)$ such that there exists $w\in\mathcal{H}^{1,\varphi}(\mathbb{R}^3)$ with $w\vert_{\Sigma}=f$, then
$$\int_{\Sigma}\, e^{\varphi} \vert\nabla f\vert^2\, d\Sigma\geq C_{\mathcal{P}}(\Sigma)C_{\varphi}(\Sigma).$$
\end{proposition}
We are ready to prove our main result. 
\subsection{Proof of Theorems A and B}
\begin{proof}[Proof of Theorem B]
Consider $\eta_{n}$ defined in \eqref{etan}. We can take the following sequence of functions $w_{n}$ with compact support in $\Sigma\cap B(0,r_{n})$ for some divergent  $r_{n}>0$  and such that $w_{n}\vert_{\Sigma}=\eta_{n}f=f_{n}$. Define,
$$a_{r_n}=-\dfrac{\int_{\Sigma\cap B(0,r_{n})}e^{\varphi}f_{n}\, d\Sigma}{\int_{\Sigma\cap B(0,r_{n})}e^{\varphi}\eta_{n}\, d\Sigma}.$$
Notice that, from Proposition \ref{propH2}, $a_{r_{n}}\rightarrow 0$ since $f\in\mathcal{H}^{1,\varphi}_{0}(\Sigma)$. Next, define $\widetilde{f}_{n}=f_{n}+a_{n}\eta_{n}$ and $\widetilde{w}_{n}=w_{n}+a_{n}\eta_{n}$. It is clear that $\widetilde{w}_{n}\vert_{\Sigma}=\widetilde{f}_{n}$ and by straighforward computation, we get that
$$0=\int_{\Sigma\cap B(0,r_{n})}e^{\varphi}\widetilde{f}_{n}\, d\Sigma=\int_{\Sigma}e^{\varphi}\widetilde{f}_{n}\chi_{r_{n}}\, d\Sigma=\int_{\Sigma}\, e^{\varphi} f\, d\Sigma. $$
for any $r_{n}>0$ and  where $\chi_{r_{n}}$ is the characteristic function on $\Sigma\cap B(0,r_{n})$. In particular, up to constant, we get that $\widetilde{f}_{n}\chi_{r_{n}}\in\mathcal{H}^{1,\varphi}_{0}(\Sigma)$. Consequently,
$$\int_{\Sigma}e^{\varphi}\vert\nabla\widetilde{f}_{n}\chi_{r_{n}}\vert^2\, d\Sigma\geq\mu\int_{\Sigma}e^{\varphi}(f_{n}+a_{n}\eta_{n})^2\chi_{r_{n}}\, d\Sigma.$$
for each $n\in\mathbb{N}$.  Thus, the proof follows taking limit as $n\rightarrow+\infty$ since $f\in\mathcal{H}^{1,\varphi}_{0}(\Sigma)$ and $\eta_{n}\rightarrow 1$ uniformly in norm $\mathcal{H}^{1,\varphi}.$
\end{proof}
\begin{proof}[Proof of Theorem A]
A nice consequence of this result gives us an estimate of $\lambda_{1}(\Delta^{\varphi})$ when we pose the Dirichlet problem for functions $w$ in the Hilbert $\mathcal{H}^{1,\varphi}_{c} (K)$ of compact support in $K$( a domain of $\mathbb{R}^3$ with $K\cap\Sigma\neq\emptyset$ an open connected subset of $\Sigma$).  Thus, $f$ has compact support in $\Sigma\cap K$. Hence, from the lower bound obtained by P. Li, S.T. Yau \cite{LY} and taking $\lambda=\text{min}\{ \lambda_{1}(\Omega\cap K),\lambda_{1}(\widetilde{\Omega}\cap K)\}$, where $\lambda_{1}$ denotes the first eigenvalue of the usual Laplacian $\Delta$, we get the following Faber-Krahn type inequality
$$\int_{\Sigma}e^{\varphi}\, \vert\nabla f\vert^2\, d\Sigma\geq \frac{3}{5}C_{\varphi}(\Sigma)\left(\frac{1}{\text{Vol}(K) }\right)^{2/3}$$
where $\text{Vol}(K)$ denotes the volume of $K$ in $\mathbb{R}^3$.
\end{proof}

\section{Some topological consequences}
The proof of Theorem C is a direct consequence of the following results.
\\

The first one is an application of the well-known result of Fisher-Colbrie on stability as local minimum of the area function $\mathcal{A}^{\varphi}$. Precisely,
\begin{corollary}
\label{Cor1}
Let $\Sigma$ be a properly embedded $[\varphi,\vec{e}_{3}]$-minimal surface in $\mathbb{R}^3$ with asymptotic flat ends, norm of the second fundamental form bounded and weight function $\varphi$ verifying \eqref{segcond} with almost quadratic growth. If
\begin{equation}
\label{controlK}
2\int_{\Sigma}e^{\varphi}K\, f^2\, d\Sigma-C_{\mathcal{P}}(\Sigma)C_{\varphi}(\Sigma)\geq\text{sup}_{\Sigma}H^2-\text{inf}_{\Sigma}\ddot{\varphi},
\end{equation} 
for any $f\in\mathcal{H}^{1,\varphi}_{0}(\Sigma)$, then $\Sigma$ is stable.
\end{corollary}

\begin{proof}
From the definition of the first eigenvalue $\lambda_{1}(\mathcal{L}^\varphi)$ together with Theorem \eqref{mainresult}, we get that
$$\lambda_{1}(\mathcal{L}^{\varphi})\geq C_{\mathcal{P}}(\Sigma)C_{\varphi}(\Sigma)-\int_{\Sigma}\, e^{\varphi}(\vert\mathcal{S}\vert^2-\ddot{\varphi}\eta^2)f^2\, d\Sigma.$$
for any $f\in\mathcal{H}^{1,\varphi}_{0}$. Now, taking into account the Gauss's formula $\vert\mathcal{S}\vert^2=H^2-2K$  we get that
$$\lambda_{1}(\mathcal{L}^\varphi)\geq C_{\mathcal{P}}(\Sigma)C_{\varphi}(\Sigma)-\int_{\Sigma}\, e^{\varphi}H^2\, f^2\, d\Sigma+\int_{\Sigma}e^{\varphi}\ddot{\varphi}\eta^2\, f^2\, d\Sigma+2\int_{\Sigma}\, e^{\varphi} K\, f^2\, d\Sigma.$$
and the proof follows from the condition \eqref{controlK}.
\end{proof}

On the other hand, if $\Sigma$ has finite topoloty and finite total curvature, we can apply the Cohn'Vossen's inequality to obtain,
\begin{corollary}
\label{Cor2}
Let $\Sigma$ be a properly embedded $[\varphi,\vec{e}_{3}]$-minimal surface in $\mathbb{R}^3$ with asymptotic flat ends, norm of the second fundamental form bounded and function $\varphi$ verifying \eqref{segcond} with almost quadratic growth. If $\Sigma$ has finite total Gauss curvature , finite topology and for any $f\in\mathcal{H}^{1,\varphi}_{0}(\Sigma)$ we get that
$$\int_{\Sigma}(1+2e^{\varphi}f^2)K\, d\Sigma\geq 0,$$
then
$$2\pi\chi(\Sigma)+\lambda_{1}(\mathcal{L}^\varphi)\geq C_{\mathcal{P}}(\Sigma)C_{\varphi}(\Sigma)-\text{sup}_{\Sigma}H^2+\text{inf}_{\Sigma}\ddot{\varphi},$$
where $\chi(\Sigma)$ is the Euler's characteristic of $\Sigma$.
\end{corollary}

\begin{corollary}
Let $\Sigma$ be a properly embedded $[\varphi,\vec{e}_{3}]$-minimal surface in $\mathbb{R}^3$ with asymptotic flat ends, norm of the second fundamental form verifying $\vert\mathcal{S}\vert^2\leq H^2$ and weight function $\varphi$ satisfying \eqref{segcond} with almost quadratic growth. If $\Sigma$ has finite total Gauss curvature and finite topology, then
$$2\pi\chi(\Sigma)+\lambda_{1}(\mathcal{L}^\varphi)\geq C_{\mathcal{P}}(\Sigma)C_{\varphi}(\Sigma)-\text{sup}_{\Sigma}H^2+\text{inf}_{\Sigma}\ddot{\varphi},$$
where $\chi(\Sigma)$ is the Euler's characteristic of $\Sigma$.
\end{corollary}

Moreover, from the work  of X. Cheng, T. Mejia and D. Zhou \cite{CMZ}, on the non-existence of stable examples with finite weighted area, we have
\begin{corollary}
\label{Cor2}
Let $\Sigma$ be a properly embedded $[\varphi,\vec{e}_{3}]$-minimal surface in $\mathbb{R}^3$ with asymptotic flat ends, norm of the second fundamental form bounded and function $\varphi$ verifying \eqref{segcond} with almost quadratic growth. Suppose that the corresponded Ilmanen's space is complete and
$\mathcal{A}^{\varphi}(\Sigma)<+\infty$. If for any $f\in\mathcal{H}^{1,\varphi}_{0}(\Sigma)$ we get that
$$\int_{\Sigma}(1+2e^{\varphi}f^2)K\, d\Sigma\geq 0,$$
then
\begin{equation}
\label{Kestimate}
2\pi\chi(\Sigma)> C_{\mathcal{P}}(\Sigma)\, \text{sup}_{\Sigma}\ddot{\varphi}-\text{sup}_{\Sigma}H^2+\text{inf}_{\Sigma}\ddot{\varphi}
\end{equation}
\end{corollary}

\begin{proof}
\textit{Argue by contradiction}, suppose that there exists such surface $\Sigma$ with 
$$2\pi\chi(\Sigma)\leq C_{\mathcal{P}}(\Sigma)\, \text{sup}_{\Sigma}\ddot{\varphi}-\text{sup}_{\Sigma}H^2+\text{inf}_{\Sigma}\ddot{\varphi}$$
Then, from Corollary \ref{Cor2}, $\Sigma$ must be stable. However, it is a contradiction with the non-existence result of X. Cheng, T. Mejia and D. Zhou \cite{CMZ}.
\end{proof}
\begin{corollary}
Let $\Sigma$ be a properly embedded $[\varphi,\vec{e}_{3}]$-minimal surface in $\mathbb{R}^3$ with asymptotic flat ends, norm of the second fundamental form verifying $\vert\mathcal{S}\vert^2\leq H^2$ and function $\varphi$ verifying \eqref{segcond} with almost quadratic growth. If the corresponded Ilmanen's space is complete and
$\mathcal{A}^{\varphi}(\Sigma)<+\infty$, then
\begin{equation}
2\pi\chi(\Sigma)> C_{\mathcal{P}}(\Sigma)\, \text{sup}_{\Sigma}\ddot{\varphi}-\text{sup}_{\Sigma}H^2+\text{inf}_{\Sigma}\ddot{\varphi}
\end{equation}
\end{corollary}
On the other hand, applying the work of B. White \cite{W} on complete surfaces with finite total curvature, we obtain this last result
\begin{corollary}
Under the conditions as above. If for any $f\in\mathcal{H}^{1,\varphi}_{0}(\Sigma)$ we have $$\int_{\Sigma}(1-2e^{\varphi} f^2)K\, d\Sigma\leq 0,$$ then $\Sigma$ has finite total curvature. In particular,  $\Sigma$ is conformally equivalent to $M-\{p_{1},\cdots,p_{s}\}$ a compact surface with $s$ ends.
\end{corollary}
\begin{proof}
From Theorem A together with the definition of $\lambda_{1}(\mathcal{L}^{\varphi})$, we get that
$$2\int_{\Sigma}e^{\varphi}K f^2\, d\Sigma\leq \lambda_{1}(\mathcal{L}^{\varphi})+C_{\mathcal{P}}(\Sigma)\, \text{sup}_{\Sigma}+\int_{\Sigma}e^{\varphi}H^2f^2\, d\Sigma-\int_{\Sigma}e^{\varphi}\ddot{\varphi}\eta^2f^2\, d\Sigma.$$
Hence,
$$\int_{\Sigma}K\, d\Sigma\leq  \lambda_{1}(\mathcal{L}^{\varphi})+C_{\mathcal{P}}(\Sigma)\, \text{sup}_{\Sigma}+\text{sup}_{\Sigma}H^2-\text{inf}_{\Sigma}\ddot{\varphi}. $$
Thus, $\Sigma$ has finite total curvature and the proof follows from the work of B. White \cite{W} for complete surfaces of finite total curvature.
\end{proof}
Finally, we would like to conclude this work giving an explicit example for the family of $[\varphi,\vec{e}_{3}]$-minimal surfaces with $\varphi$ a quadratic polynomial, namely, we consider $\varphi(z)=\beta^2-\alpha^2\, z$ for any $\vert z\vert\leq \beta/\alpha$ and fixed  any $\alpha,\beta>0$. It is clear that, taking $C_{\varphi}(\Sigma)\in ]0,1/2\beta^2[$ the function $\varphi$ verifies the properties \eqref{segcond}. Hence, we can give this last result
\begin{corollary}
Let $\alpha,\beta>0$ be a positive real numbers and $\varphi(z)=\beta^2-\alpha^2\, z$ for any $\vert z\vert\leq \beta/\alpha$. If $\Sigma$ is a properly embedded $[\varphi,\vec{e}_{3}]$-minimal surface with asymptotic flat ends, norm of the second fundamental form $\vert\mathcal{S}\vert^2\leq H^2$, finite total Gauss curvature and finite topology, then $$2g+N(\text{ends})\leq 4+n+m,$$ where $g$ is the genus of $\Sigma$, $N(\text{ends})$ is the number of ends and $n,m$ are the only integers such that
$$\lambda_{1}(\mathcal{L}^{\varphi})\in [2n\pi, 2(n+1)\pi[ \ \ \text{and} \ \ \alpha^2(1+2\alpha\beta)\in ]0,\pi[\, \cup\, [m\pi,(m+1)\pi[.$$
In particular, if $\lambda_{1}(\mathcal{L}^{\varphi})\leq 0$ and $\alpha^2(1+2\alpha\beta)\in ]0,\pi[$, then  $$2g+N(\text{ends})\leq 3.$$
\end{corollary}

\end{document}